\pdfoutput=1
\documentclass[10pt]{amsart}
\usepackage{amssymb}
\usepackage{mathrsfs}                     
\usepackage{mathtools}
\mathtoolsset{showonlyrefs}
\usepackage[hyperindex=true,bookmarks=true,bookmarksnumbered=true]{hyperref}
\usepackage[all]{xy}
\allowdisplaybreaks

\def\undersetbrace#1\to#2{\underbrace{#2}_{#1}}
\def\oversetbrace#1\to#2{\overbrace{#2}^{#1}}
\def\AMSunderset#1\to#2{\underset{#1}{#2}}
\def\AMSoverset#1\to#2{\overset{#1}{#2}}

\swapnumbers
\newtheorem{proposition}[subsection]{Proposition}
\newtheorem*{proposition*}{Proposition}
\newtheorem{theorem}[subsection]{Theorem}
\newtheorem*{theorem*}{Theorem}
\newtheorem*{maintheorem*}{Main Theorem}
\newtheorem{maintheorem}[subsection]{Main Theorem}
\newtheorem{lemma}[subsection]{Lemma}
\newtheorem*{lemma*}{Lemma}
\newtheorem{corollary}[subsection]{Corollary}
\newtheorem*{corollary*}{Corollary}

\theoremstyle{definition}
\newtheorem*{remark*}{Remark}
\newtheorem{definition}[subsection]{Definition}
\newtheorem{remark}[subsection]{Remark}
 
\def\ign#1{}             
\def\o{\circ}
\def\X{\mathfrak X}
\def\al{\alpha}
\def\be{\beta}

\def\de{\delta}

\def\io{\iota}

\def\rh{\rho}
\def\si{\sigma}
\def\ta{\tau}
\def\ph{\varphi}

\def\ps{\psi}
\def\om{\omega}
\def\Ga{\Gamma}
\def\De{\Delta}

\def\La{\Lambda}

\def\Om{\Omega}
\def\i{^{-1}}
\def\x{\times}
\def\p{\partial}
\let\on=\operatorname
\def\L{\mathcal L}
\def\Diff{\on{Diff}}
\def\Dens{\on{Dens}}
\def\Prob{\on{Prob}}

\let\mc=\mathcal
\let\mf=\mathfrak

\usepackage{ytableau}
\ytableausetup{boxsize=1.05em,centertableaux}
\begin{document}
\title[]
{Diffeomorphism invariant tensor fields on the space of
  smooth positive densities on a compact manifold with corners}
\author{Peter W. Michor}
\address{
Peter W.\ Michor: Fakult\"at f\"ur Mathematik,
Universit\"at Wien, Os\-kar-Mor\-gen\-stern-Platz 1, A-1090 Wien, Austria.}
\email{peter.michor@univie.ac.at}
 
\date{{\today} } 
 
\keywords{Fisher--Rao Metric; Information Geometry;  Invariant Metrics; Space of Densities; 
Groups of Diffeomorphisms}
\subjclass[2010]{Primary 58B20, 58D15} 
 
\begin{abstract}
It was proved in \cite{BBM2016} that on a closed manifold of dimension greater than one,
every smooth weak Riemannian metric on the space of smooth positive probability densities
that is invariant under the action of the diffeomorphism group is a multiple of the
Fisher--Rao metric. Here we describe all purely covariant tensor fields of arbitrary order on the space
of positive densities of a compact manifold with corners which are invariant under the action of the connected component of the diffeomorphism group.  Finally, we classify the considerably fewer purely contravariant invariant tensor fields, describe the basic mixed contraction, and indicate the expected form of general invariant tensor fields.

\end{abstract}
\def\LaTeXonly{}
 
\maketitle

\section{Introduction} 
 
The Fisher--Rao metric on the space $\on{Prob}(M)$ of probability densities is invariant under the action of the diffeomorphism group $\Diff(M)$.
Restricted to finite-dimensional submanifolds of $\on{Prob}(M)$, 
so-called statistical manifolds, it is called Fisher's information metric \cite{Ama1985}. 
A uniqueness result was established \cite[p. 156]{Cen1982} for Fisher's information 
metric on finite sample spaces and \cite{AJLS2014} extended it to infinite sample spaces.       
The Fisher--Rao metric on the infinite-dimensional manifold of all positive probability densities 
was studied in \cite{Fri1991}, including the computation of its curvature. 
In \cite{BBM2016} it was proved that the Fisher--Rao metric on $\Prob(M)$ is, up to a multiplicative constant, the unique $\Diff(M)$-invariant metric on a compact manifold without boundary. In fact, all $\Diff(M)$-invariant bilinear covariant tensor fields on $\on{Dens}_+(M)$ were determined. Here we investigate compact manifolds with corners (for example simplices) and invariant purely covariant tensor fields of arbitrary order. The interior terms are described by volume-natural scalar differential operators. The boundary terms have a different intrinsic form: since the ambient volume form restricts to zero on every positive-codimension stratum, it cannot supply a boundary measure. On positive-dimensional boundary strata the boundary generators are products of restrictions of ambient \(\mu\)-natural scalar differential operators with wedges of exterior differentials \(d(\alpha_i/\mu)\). At zero-dimensional strata additional vertex-natural differential operators occur, which are classified explicitly by finite jet equations. In the final Section \ref{sec:contra} we also describe  the significantly fewer invariant purely contravariant tensor fields and the generalized contractions and indicate the expected form of general invariant tensor fields.

\subsection{The Fisher--Rao metric}
Let $M^m$ be a smooth compact connected manifold without boundary. 
Let $\on{Vol}(M)\to M$ be 
the line bundle of smooth densities whose cocycle of transition functions is 
$|\det(d(u_a\o u_b\i))|\i\o u_b$ for any smooth atlas $(u_a:U_a \to \mathbb R^m)_{a\in A}$;  for more details 
 we refer to \cite{BBM2016}. Moreover, we denote by $|\quad|:\La^mT^*M\to \on{Vol}(M)$ the fiber respecting absolute value mapping. 
We let $\on{Dens}_+(M)$ denote the space of smooth positive densities on $M$, i.e., 
$\on{Dens}_+(M) = 
\{ \mu \in \Ga(\on{Vol}(M)) \,:\, \mu(x) > 0\; \forall x \in M\}$. Let $\on{Prob}(M)$ be the 
subspace of positive densities with integral 1 on $M$. Both spaces are smooth Fr\'echet manifolds, 
in particular they are open subsets of the affine spaces of all densities or densities of integral 
1, respectively. For $\mu \in \on{Dens}_+(M)$ we have    
$ T_\mu \on{Dens}_+(M) = \Ga(\on{Vol}(M))$ and for $\mu\in \Prob(M)$ we have 
$$
T_\mu\Prob(M)=\{\al\in \Ga(\on{Vol}(M)): \int_M\al =0\}.
$$
The Fisher--Rao metric is a Riemannian metric on $\on{Prob}(M)$ and is defined as follows:
$$
G^{\operatorname{FR}}_\mu(\al,\be) = \int_M \frac{\al}{\mu}\frac{\be}{\mu}\mu.
$$
This metric is invariant under the associated action of $\Diff_0(M)$ on $\Prob(M)$, since
$$
\Big((\ph^*)^*G^{\operatorname{FR}}\Big)_\mu(\al,\be) = G^{\operatorname{FR}}_{\ph^*\mu}(\ph^*\al,\ph^*\be) 
= \int_M \Big(\frac{\al}{\mu}\o \ph\Big)\Big(\frac{\be}{\mu}\o \ph\Big)\ph^*\mu
= \int_M \frac{\al}{\mu}\frac{\be}{\mu}\mu\,.
$$
 
The uniqueness result for the Fisher--Rao metric follows from the following classification of 
$\on{Diff}(M)$-invariant bilinear forms on $\on{Dens}_+(M)$.

\begin{theorem} {\rm \cite{BBM2016}} Let $M$ be a compact manifold without boundary of dimension $\geq 2$.
Let $G$ be a smooth (equivalently, bounded) bilinear form on $\Dens_+(M)$ which is invariant under the action of 
$\Diff_0(M)$. Then 
$$
G_\mu(\al,\be)=C_1(\mu(M)) \int_M \frac{\al}{\mu}\frac{\be}{\mu}\,\mu + C_2(\mu(M))  \int_M\al \cdot \int_M\be
$$
for some smooth functions $C_1,C_2$ of the total volume $\mu(M)$.  
\end{theorem}
 
\subsection{Acknowledgments} This paper was started 2016 after \cite{BBM2016} was finished, and a first step toward the main result was done in \cite{BMPR18}. Then the work on this paper hit a wall, until in 2026 discussions  with Claude Sonnet 5 (free version, then Pro version) 
and later discussions with ChatGPT (GPT-5.6 Sol, OpenAI) helped me to identify and clarify several proof issues, in particular around Lemma \ref{old8} and the proofs in Section \ref{sec:contra}.
 Thanks are due to  Boris Kruglikov, Valentin Lychagin, Martin Bauer, and Philipp Harms for discussions and hints.  
  
\section{Manifolds with corners}
 
\subsection{Manifolds with corners alias quadrantic (orthantic) manifolds}
For more information we refer to \cite{DouadyHerault73}, \cite{Michor80}, \cite{Melrose96}, etc.
Let $Q=Q^m=\mathbb R^m_{\ge 0}$ be the positive orthant or quadrant. By Whitney's extension theorem or Seeley's theorem,
restriction $C^{\infty}(\mathbb R^m)\to C^{\infty}(Q)$ is a surjective continuous linear mapping which admits a continuous linear section (extension mapping); so $C^{\infty}(Q)$ is a direct summand in $C^{\infty}(\mathbb R^m)$. A point $x\in Q$ is called a \emph{corner of codimension} $q>0$ if $x$ lies in the intersection of $q$ distinct coordinate hyperplanes. Let $\p^q Q$ denote the set of all corners of codimension $q$.
 
A manifold with corners (recently also called a quadrantic manifold) $M$ 
is a smooth manifold modelled on open subsets of $Q^m$.
We assume that it is connected and second countable; then it is paracompact and for each open cover it admits a subordinated smooth partition of unity. Any manifold with corners $M$ is a submanifold with corners of an open manifold $\tilde M$ of the same dimension, and each smooth function on $M$ extends to a smooth function on $\tilde M$. Moreover,  restriction $C^\infty(\tilde M)\to C^\infty(M)$ is a surjective continuous linear map which admits a continuous linear section; this follows by  gluing via a smooth partition of unity  from the result about quadrants. Thus $C^\infty(M)$ is a topological direct summand in 
$C^\infty(\tilde M)$ and the same holds for the dual spaces: The space of distributions 
$\mathcal D'(M)$, which we identify with $C^\infty(M)'$ in this paper, is a direct summand in 
$\mathcal D'(\tilde M)$. It consists of all distributions with support in $M$.
 
Proofs will be mainly done for  oriented  $M$; the transition to non-orientable $M$ will be done by using the orientable double cover in \ref{nonorientable}. Eventually we will assume that $M$ is compact. Diffeomorphisms of $M$ map the boundary $\p M$ to itself and map  the boundary $\p^q M$ of corners of codimension $q$ to itself; $\p^q M$ is a submanifold of codimension $q$ in $M$; in general $\p^q M$ has finitely many connected components. We shall consider $\p M$ as stratified into the connected components of all $\p^q M$ for $q>0$. Note that $\p^0 M = M\setminus \p M$ is the open interior of $M$.
 
Each diffeomorphism of $M$ restricts to a diffeomorphism of $\p M$ and to a diffeomorphism of each $\p^q M$. The Lie 
algebra of $\Diff_0(M)$ consists of all vector fields $X$ on $M$ such that $X|\p^q M$ is tangent to 
$\p^q M$. We shall denote this Lie algebra by $\X(M,\p M)$.
 
\subsection{Differential forms}
There are several differential complexes on a manifold with corners.  
If $M$ is not compact there are also the versions with compact support. 
\begin{itemize}
\item Differential forms that vanish near $\p M$. If $M$ is compact, this is the same as
the differential complex $\Om_c(M\setminus \p M)$ of differential forms with compact support 
in the open interior $M\setminus \p M$. 
\item $\Om(M,\p M) = \{\al\in \Om(M): j_{\p^q M}^*\al =0 \text{ for all } q\ge 1\}$, the complex of differential forms that pull back to 0 on each boundary stratum. Note that $\Om^0(M,\p M) = \{f \in C^{\infty}(M): f|_{\p M} = 0\}$.
\item $\Om(M)$, the complex of all differential forms. Its cohomology equals 
singular cohomology with real coefficients of $M$, since $\mathbb R\to \Om^0\to \Om^1\to \dots$
is a fine resolution of the constant sheaf on $M$; for that one needs existence of smooth partitions of unity and the Poincar\'e lemma which holds on manifolds with corners.
The Poincar\'e lemma can be proved as in \cite[9.10]{Mic2008} in each quadrant.
\item $\Om(M,\p M)$ is a graded ideal in $\Om(M)$. 
\end{itemize}
If $M$ is an oriented manifold with corners of dimension $m$ and if $\mu\in \Om^m(M)$ is a nowhere vanishing form of top degree, then $\X(M)\ni X\mapsto i_X\mu\in \Om^{m-1}(M)$ is a linear isomorphism. 
Moreover, $X\in \X(M,\p M)$ (tangent to the boundary) if and only if $i_X\mu\in\Om^{m-1}(M,\p M)$.
 
Let us consider the short exact sequence of differential graded algebras
$$
0\to \Om(M,\p M) \to \Om(M) \to \Om(M)/\Om(M,\p M)\to 0\,.
$$
The complex $\Om(M)/\Om(M,\p M)$ is a subcomplex of the product  of $\Om(N)$ for all connected components $N$ of all 
$\p^q M$. The quotient consists of forms which extend continuously over boundaries to $\p M$ with its induced topology in such a way that one can extend them to smooth forms on $M$; this is contained in the space of `stratified forms' as used in \cite{Valette15}. There Stokes' formula is proved for stratified forms.
 
\begin{proposition}[Stokes' theorem] \label{Stokes}
For a connected oriented manifold $M$ with corners of dimension $\dim(M)=m$ and for any $\om\in\Om^{m-1}_c(M)$ we have
$$
\int_M d\om = \int_{\p^1M} j_{\p^1 M}^*\om\,.
$$
\end{proposition}
 
See \cite{BMPR18} for a short proof.
 
\subsection{Top cohomology of the pair $(M,\p M)$}
For a connected oriented manifold with corners $M$ of dimension $m$ (we assume that $\p M$ is not empty) we consider the following diagram; see \cite[section 8]{BMPR18} for the simple proofs. 
Here $\om\in\Om^m_c(M\setminus \p M)$ is a fixed form with $\int\om = 1$.
All instances of $\mathbb R$ in the diagram are connected by identities which fit commutingly into the diagram.
Each line is the definition of  the corresponding top de~Rham cohomology space. 
The integral in the first line induces an isomorphism in cohomology since 
$M\setminus \p M$ is a connected oriented open manifold.
The bottom triangle commutes by Stokes' theorem \ref{Stokes}.
$$
\xymatrix{
\Om^{m-1}_c(M\setminus \p M) \ar[r]^{d} \ar@{^{(}->}[d] & 
\Om^m_c(M\setminus \p M) \ar@/^1.5pc/[rr]^{\int_{M\setminus \p M}} \ar@{->>}[r]  \ar@{^{(}->}[d] &
H^m_c(M\setminus \p M) \ar@{=}[r]  \ar[d]& \mathbb R
\\
\Om^{m-1}_c(M, \p M) \ar[r]^{d} \ar@{^{(}->}[d] & 
\Om^m_c(M,\p M) \ar@{->>}[r] \ar@{=}[d] \ar@/_1.5pc/[rr]_{\qquad\qquad\qquad\qquad\quad\int_M} &
H^m_c(M,\p M) \ar@{=}[r] \ar[d]&  \mathbb R
\\
\Om^{m-1}_c(M) \ar[r]^{d} \ar[dr]_{\int_{\p^1 M}\o j_{\p^1 M}^*} & 
\Om^m_c(M)  \ar@{->>}[r] \ar[d]^{\int_M}  &
H^m_c(M) \ar@{=}[r]& 0
\\
& \mathbb R  
}
$$
 
\begin{theorem} {\rm (\cite{BMPR18} Moser's theorem for manifolds with corners)}\label{Moser}
Let $M$ be a compact smooth manifold with corners, possibly non-orientable. 
Let $\mu_0$ and $\mu_1$ be two smooth positive densities in  
$\Dens_+(M)$ with $\int_M\mu_0 = \int_M\mu_1$.
Then there exists a diffeomorphism $\ph:M\to M$ diffeotopic to the identity such that 
$\mu_1= \ph^*\mu_0$. If and only if $\mu_0(x)=\mu_1(x)$ for each corner $x\in\p^{\ge 2}M$ 
of codimension $\ge 2$, then $\ph$ can be chosen to be the identity on $\p M$. 
\end{theorem}

\section{Diffeomorphism invariant tensor fields on the space of densities}

\begin{definition}\label{old8def}
Let $\mu$ be a volume form on an oriented manifold $M$. A multilinear differential operator
$L:C^\infty(M)^r\to C^\infty(M)$ is called \emph{natural for $\mu$} if it is equivariant under local $\mu$-preserving diffeomorphisms. Equivalently, for every vector field $X$ on an open set with $\L_X\mu=0$ we have
\begin{equation}\label{eq:equiv}
\sum_{j=1}^r L(f_1,\dots,Xf_j,\dots,f_r)
 = X\big(L(f_1,\dots,f_r)\big).
\end{equation}
The Jacobians
$J_S(f_{1},\dots,f_{m})=df_{1}\wedge\dots\wedge df_{m}/\mu$
are natural, and products and insertions $L_1(\dots,L_2(\dots),\dots)$ of $\mu$-natural operators are again $\mu$-natural. The symbol spaces of all $\mu$-natural operators are described by Young tableaux together with the higher-jet equations in Proposition \ref{old8prop} below. 
\\
The space of all $\mu$-natural $r$-linear operators is denoted $\mc N_r(M,\mu)$. 
\end{definition}

\begin{maintheorem}\label{maintheorem}
Let $M$ be an oriented compact connected manifold with corners, of dimension $m\ge2$, and write
\[
\p^pM=(\p^pM)_1\sqcup\cdots\sqcup(\p^pM)_{n_p}
\]
for the connected components of the codimension-$p$ stratum.  Then the associative algebra of bounded
$\Diff_0(M)$-invariant purely covariant tensor fields on $\on{Dens}_+(M)$ is generated by the following tensor fields.
For $\mu\in\on{Dens}_+(M)$ and $\al_i\in T_\mu\on{Dens}_+(M)=\Ga(\on{Vol}(M))$ put
$f_i=\al_i/\mu$.
\begin{align*}
& h(\mu(M)),\qquad h\in C^\infty(\mathbb R_{>0},\mathbb R),\\
& \int_M f_1\cdots f_n\,\mu,\qquad n\ge1,\\
& \int_M \Big(\prod_{i\in S_0}f_i\Big)
 F_\mu(f_s:s\in T)\,\mu,
 \quad S_0\sqcup T=\{1,\dots,n\},
\end{align*}
where $F_\mu$ is any $\mu$-natural multilinear scalar differential operator.  For
$1\le p<m$, $1\le j\le n_p$, put $S=(\p^pM)_j$ and $q=m-p>0$.  The positive-dimensional boundary generators are
\begin{equation}\label{eq:boundarygenerator}
 \int_S
 \Big(\prod_{i\in S_0}f_i\Big)
 F_\mu(f_s:s\in T)\,
 df_{i_1}|_S\wedge\cdots\wedge df_{i_q}|_S,
 \quad
 S_0\sqcup T\sqcup\{i_1,\ldots,i_q\}=\{1,\ldots,n\},
\end{equation}
where $F_\mu$ is an ambient $\mu$-natural multilinear scalar differential operator.
Finite sums of such expressions, and products of the resulting tensor fields, are understood.
At a vertex $v\in\p^mM$ one must in addition allow the vertex generators
\begin{equation}\label{eq:vertexgenerator}
 \mathcal V_{\mu,v}(f_1,\ldots,f_n),
\end{equation}
where $\mathcal V_{\mu,v}$ is any multilinear scalar differential operator on ambient jets at $v$
which is natural for the germ of the face-preserving, $\mu$-preserving corner pseudogroup fixing $v$.
Equivalently, after choosing adapted volume coordinates at $v$, its polynomial symbol is characterized
by the finite system of vertex equations in Proposition \ref{vertexclassification}.
These vertex operators are genuinely larger than restrictions of ambient $\mu$-natural scalars in general.
\end{maintheorem}

The proof will occupy the rest of this section.

\begin{corollary}\label{nonorientable}
For a non-orientable compact manifold $\overline M$ with corners, let $\pi:M \to \overline M$ be its orientable double cover with its deck transformation $\ta:M\to M$. We consider the bounded linear isomorphism
$$\tfrac12 \pi^*:\Dens_+(\overline M) \to \{\al\in \Dens_+(M): \ta^*\al=\al\} \subset\Dens_+(M).$$
Then the set of $\ta^*$-invariant  generators for the algebra of bounded $\Diff_0(M)$- and $\ta$-invariant purely covariant tensor fields on $\Dens_+(M)$, applied to $\tfrac12\pi^* \al_i$ and $\tfrac12\pi^*\mu$ for $\mu\in \Dens_+(\overline M)$ and $\al_i\in T_\mu\Dens_+(\overline M)$, is a set of generators for the algebra of $\Diff_0(\overline M)$-invariant bounded tensor fields on $\Dens_+(\overline M)$. \qed
\end{corollary}

\subsection{Orienting all boundary components}\label{ssec:orientation}
For an oriented compact connected manifold $M$ with corners, of dimension $m\ge2$, fixing a numbering
\[
\p^pM=(\p^pM)_1\sqcup\cdots\sqcup(\p^pM)_{n_p}
\]
for all connected components of the codimension-$p$ stratum for $1\le p\le m$ allows us to induce an orientation on each $(\p^p M)_j$ as follows: (1) Each $(\p^1 M)_j$ is a codimension 1 boundary of $M\setminus \p M$ and thus inherits an orientation. (2) Each $(\p^2 M)_k$ is a codimension 1 boundary of an $(\p^1 M)_j$ for the lowest number $j$ and inherits thus an orientation. Note that this orientation depends on $j$. And so on. 

\subsection{Reduction to local invariant differential functionals.}
For the structural results below it suffices to treat the orientable case. Let us fix a basic probability density (a  volume form) $\mu_0$. By 
Moser's theorem \ref{Moser} for manifolds with corners, for each $\mu\in\Dens_+(M)$ there exists a 
diffeomorphism $\ph_\mu\in \Diff_0(M)$ with $\ph_\mu^*\mu=\mu(M)\mu_0=:c.\mu_0$ where 
$c=\mu(M)=\int_M\mu>0$.
Then 
$$
\big((\ph_\mu^*)^*G\big)_\mu(\al_1,\dots,\al_n) = G_{\ph_\mu^*\mu}(\ph_\mu^*\al_1,\dots,\ph_\mu^*\al_n) =  
G_{c.\mu_0}(\ph_\mu^*\al_1,\dots,\ph_\mu^*\al_n)\,.
$$
Thus it suffices to show that for any $c>0$ we have
\[
G _{c\mu_0}(\al_1,\dots,\al_n)= C_0(c).\int_M\frac{\al_1}{\mu_0}\dots\frac{\al_n}{\mu_0}\mu_0 + \dots
\]
for some functions $C_0,\dots$ of the total volume $c = \mu(M)$. 
Since $c\mapsto c.\mu_0$ is a smooth curve in $\Dens_+(M)$,
the functions $C_0,\dots$ are then smooth in $c$. 
Both $k$-linear forms are still invariant under the action of the group
$\Diff_0(M,c\mu_0)=\Diff_0(M,\mu_0)=\{\ps\in \Diff_0(M): \ps^*\mu_0=\mu_0\}$.
The $n$-linear form 
$$
\big(T_{\mu_0}\Dens_+(M)\big)^k\ni(\al_1,\dots,\al_n)\mapsto 
G_{c\mu_0}\Big(\frac{\al_1}{\mu_0}\mu_0,\dots,\frac{\al_n}{\mu_0}\mu_0\Big)
$$
can be viewed as a bounded $n$-linear form  
\begin{gather*}
C^\infty(M)^k\ni (f_1,\dots,f_k)\mapsto\big((f_{k+1},\dots, f_n)\mapsto  G_c(f_1,\dots,f_n)\big)\,.
\end{gather*}
Using the Schwartz kernel theorem \cite[Theorem 5.2.1]{Hor1983}, $G_c$ has a kernel $\hat G_c$, which is a distribution (generalized function) in 
\begin{align*}
\mathcal D'(M^n) \cong \mathcal D'(M)\,\bar\otimes\dots\bar\otimes\, \mathcal D'(M) &= 
\big(C^\infty(M)\,\bar\otimes\dots\bar\otimes\, C^\infty(M)\big)' 
\\&
\cong L(C^\infty(M^k), \mathcal D'(M^{n-k}))\,.
\end{align*}
with defining relations
\begin{align*}
G_c(f_1,\dots,f_n) &= \langle \check G_c(f_1,\dots,f_k),f_{k+1}\otimes\dots\otimes f_n \rangle 
= \langle \hat G_c, f_1\otimes\dots\otimes f_n \rangle\,,
\end{align*}
and permuted versions thereof. 
Moreover, $\hat G_c$ is invariant under the diagonal action of $\Diff_0(M,\mu_0)$ on $M^n$. In view 
of the tensor product in the defining relations, the 
infinitesimal version of this invariance is:  
\begin{align*}
0&=\langle\L_{X^{\text{diag}}}\hat G_c,f_1\otimes\dots\otimes f_n\rangle
=-\langle\hat G_c,\L_{X^{\text{diag}}}(f_1\otimes\dots\otimes f_n)\rangle 
\\&
= - \sum_{i=1}^n\langle\hat G_c,f_1\otimes\dots\otimes\L_X f_i\otimes\dots\otimes f_n)\rangle 
\\
X^{\text{diag}} &= X\x 0\x \dots\x 0 + 0\x X\x 0\x \dots \x 0 + \dots.
\end{align*}
for all $X\in\X(M,\mu_0)$.
 
We will consider various (permuted versions) of the associated bounded mappings
$$
\check G_c:C^\infty(M)^k \to \big(C^\infty(M)^{n-k}\big)' = \mathcal D'(M^{n-k})\,.
$$
We shall use the fixed density $\mu_0\in\on{Dens}_+(M)$ for the rest of this section. So we identify distributions on $M^k$ with the dual space $C^\infty(M^k)'=:\mathcal D'(M^k)$
 
\subsection{Good subsets of $M$}\label{good}
Consider an open oriented connected subset $U\subset M$ which is diffeomorphic to an open $m$-ball intersected with a quadrant, such that $U\cap (\p^q M)_j$  is
connected for each connected component of a boundary stratum. We shall call such a subset $U$ a \emph{good subset}.
Each point in $M$ admits a good open neighborhood.
On a good subset $U$ each density $\mu$ is an $m$-form where $m=\dim(M)$. 
Let $j_{\p^q M} = j:\p^q M\to M$ be the embedding of the boundary stratum of codimension $q$. Let $\p^q U= U\cap\p^q M$. 
Then $(j|\p U)^*\mu = 0$ since 
$\dim(\p^q M)=m-q<m$. Denote by $\Om^k(U,\p U)$ the space of all $k$-forms $\om\in \Om^k(U)$ such 
that $(j|\p^q U)^*\om=0$ for all $1\le q\le m$, and by $\Om^k_c(U,\p U)$ those with compact support. 
The mapping $\hat\io_{\mu_0}: \X(U)\to \Om^{m-1}(U)$ given by $X\mapsto i_X\mu_0$ is an isomorphism. 
Moreover, $\hat\io_{\mu_0}(X) \in  \Om^{m-1}(U,\p U)$ if and only if  $X|\p^q U$ is tangent to 
$\p^q U$ for all $q>0$; i.e., if $X\in \X(U,\p U)$. 
 
For a good subset $U$ with non-empty boundary we have $H^m_c(U)=0$, and $H^m_c(U,\p U)=\mathbb R$ induced by integration; see \cite[section 8]{BMPR18}. Moreover we also have $H^{m-1}(U,\p U) = 0$ since the good subset  $U$ is contractible in a way which contracts each boundary stratum of $\p U$ in itself. 
 
\subsection{The Lie algebra of $\Diff(M,\mu_0)$}\label{old1}
For a fixed positive density $\mu_0$ on $M$,
the Lie algebra of $\Diff(M,\mu_0)$ which we will denote by 
$\X(M,\p M,\mu_0)$, is the subalgebra of vector fields which are tangent to each boundary stratum and which are divergence free: $0=\on{div}^{\mu_0}(X) := \frac{\L_X\mu_0}{\mu_0}$.
These are exactly the fields $X$ such that for each good subset $U$ (where each density can be identified with an $m$-form) the form 
$\hat\io_{\mu_0}(X)$ is a closed form in $\Om^{m-1}(U,\p U)$, equivalently  
$0=\on{div}^{\mu_0}(X) := \frac{\L_X\mu_0}{\mu_0}$.
 
\subsection{The set of vector fields $\X_{\text{exact}}(M,\p M,\mu_0)$}\label{old3}
Denote by $\X_{\text{exact}}(M,\p M,\mu_0)$ the set (not a vector space) of `exact' divergence free 
vector fields $X = \hat\io_{\mu_0}\i(d\om)$, where $\om\in\Om^{m-2}_c(U,\p U)$ for a good
subset $U\subset M$. They are automatically tangent to each boundary stratum since $d\om\in \Om^{m-1}_c(U,\p U)$.                                                                                                                                                                                                                                                                          
 
For $x\in \p^p U$ and $g\in C^{\infty}_c(U)$, $g=1$ near $x,$ 
using coordinates $u^1\ge0,\dots,u^p\ge 0, u^{p+1},\dots,u^m$ near $x$ on $U$ with 
$\mu_0=du^1\wedge \dots\wedge du^m$,  for $k<\ell$ and $p<\ell$, the vector field 
\begin{equation*}
X:= \hat\io_{\mu_0}\i d(g.u^k.du^1\wedge  \dots \widehat{du^k}\dots 
\widehat{du^\ell}\dots\wedge du^m)\in \X_{\text{exact}}(M,\p M,\mu_0)
\end{equation*}
satisfies $X = \pm\p_{u^\ell}$ near $x$, so we can produce any frame in $T_x\p^p M$. 
  
Now let $A=(A^i_j)\in \mathfrak{sl}(m,\mathbb R)$ be a constant matrix with $\sum_iA^i_i=0$.  At a boundary point $x\in\p^pM$, tangency merely to the codimension-$p$ stratum $\{u^1=\cdots=u^p=0\}$ would only force the upper-right block to vanish, but in every neighborhood of $x$ there are points in $\p^{p-1}M$.  The Lie algebra $\X(M,\p M,\mu_0)$ consists of vector fields tangent to every local boundary hypersurface.  Hence for each $i\le p$ the $i$-th component has to vanish on the face $\{u^i=0\}$, and for the linear field
$\bar X_A = \sum_{i,j}A^i_j u^j \p_{u^i}$ this is equivalent to
$A^i_j=0$ for $i\le p$ and $j\ne i$.  Thus
\[
A= \begin{pmatrix}
\operatorname{diag}(a_1,\dots,a_p) & 0 \\ A_3 & A_2
\end{pmatrix},\qquad
\sum_{i=1}^p a_i+\operatorname{Tr}(A_2)=0,
\]
while the lower-left block $A_3$ is unrestricted.  This is also the infinitesimal form of the linear corner isotropy used in Proposition \ref{boundarytableaux}.  In that case $\bar X_A$
satisfies
\begin{align*}
d\,i_{\bar X_A} \mu_0 &= d\big((-1)^{i-1} \sum_{i,j} A^i_j u^j.du^1\wedge \dots \widehat{du^i}\dots \wedge du^m\big) 
\\&
= (-1)^{i-1} \sum_{i,j} A^i_j. \de^j_i du^1\wedge\dots \wedge du^m =0 \,.
\end{align*}
Since $U$ is a good subset, $H^{m-1}(U,\p U)=0$ as noted in  \ref{good}, thus 
$i_{\bar X_A} \mu_0 = d\al$ for some $\al\in \Om^{m-2}(U,\p U)$. 
But then $X= \hat \io^{-1} d (g.\al) \in \X_{\text{exact}}(M,\p M,\mu_0)$ and $X=\bar X$ near $x$.
 
But note that $u^i\p_{u^i}$ near $x$ is not in $\X(U,\p U,\mu_0)$: we have 
$\L_{u^i\p_{u^i}}\mu_0= \mu_0 $.

\begin{lemma} \label{old2}
If for $f\in C^\infty(M)$ and a good set $U\subseteq M$ we have 
$(\L_Xf)|U=0$ for all 
$X\in\X_{\text{exact}}(M,\p M,\mu_0)$, then $f|U$ is constant.
\end{lemma}

\begin{proof}
Let $x \in U\setminus \p U$. For every tangent vector $X_x\in T_xM$ we can find a vector field 
$X\in\X_{\text{exact}}(M,\p M,\mu_0)$ such that $X(x)=X_x$ as shown in \ref{old3}.
Thus $\L_X f|_{U\setminus \p U} = 0$ for all $X \in \X_{\text{exact}}(M,\p M,\mu_0)$ 
implies $df = 0$ and hence $f$ is constant on $U\setminus \p U$, and by continuity also on $U$. 
\end{proof}

\begin{lemma}\label{old4}
Let $A\in\mathcal D'(M)$ satisfy $\L_XA=0$ for every
$X\in\X_{\mathrm{exact}}(M,\p M,\mu_0)$.  Then there are a constant $C$ and constants
$c_v$, $v\in\p^mM$, such that
\[
 A=C\mu_0+\sum_{v\in\p^mM}c_v\,\delta_v.
\]
Equivalently, $A=C\mu_0$ on $M\setminus\p^mM$, and the only invariant distributions
supported at a vertex are multiples of point evaluation.
\end{lemma}

\begin{proof}
We first prove the assertion away from the vertices.  Let $x\in\p^pM$ with $0\le p<m$ and choose a good coordinate neighborhood $U_x\cong Q_p^m$ with
$\mu_0=du^1\wedge\cdots\wedge du^m$.  If $g\in C_c^\infty(U_x)$ has
$\int_Mg\mu_0=0$, then $[g\mu_0]=0$ in
$H_c^m(U_x,\p U_x)=\mathbb R$, so there is
$\alpha\in\Omega_c^{m-1}(U_x,\p U_x)$ with $d\alpha=g\mu_0$.

\noindent\textbf{Claim.}  For $p<m$ one can write
\[
 \alpha=\sum_\nu \bar f_\nu\,d\beta_\nu,
 \qquad
 \bar f_\nu\in C_c^\infty(U_x),\quad
 \beta_\nu\in\Omega^{m-2}(U_x,\p U_x).
\]
Indeed, since $j_{\p U_x}^*\alpha=0$ we may write
\[
\alpha=
 \sum_{k\le p} f_k u^k\,du^1\wedge\cdots\widehat{du^k}\cdots\wedge du^m
 +\sum_{j>p}f_j\,du^1\wedge\cdots\widehat{du^j}\cdots\wedge du^m .
\]
Because $p<m$, the tangential coordinate $u^{p+1}$ is available.  For $k\le p$ the form
$u^ku^{p+1}du^1\wedge\cdots\widehat{du^k}\cdots
\widehat{du^{p+1}}\cdots\wedge du^m$ belongs to
$\Omega^{m-2}(U_x,\p U_x)$, and differentiating it expresses the $k$-th summand above as a linear combination of terms $\bar f\,d\beta$; the extra term contains the factor $u^{p+1}$ and is treated in the same way using
$d(u^1du^2\wedge\cdots\widehat{du^{p+1}}\cdots\wedge du^m)$.
The terms with $j>p$ are directly of the form
$f_jd(u^1du^2\wedge\cdots\widehat{du^j}\cdots\wedge du^m)$.
This proves the claim.  (The argument deliberately uses $p<m$ and is not asserted at a vertex.)

Choose $h\in C_c^\infty(U_x)$ equal to $1$ on the supports of all $\bar f_\nu$.  Then
$\alpha=\sum_\nu\bar f_\nu d(h\beta_\nu)$ and
$X_\nu:=\hat\iota_{\mu_0}^{-1}d(h\beta_\nu)$ lies in
$\X_{\mathrm{exact}}(M,\p M,\mu_0)$.  Hence
\[
 g=\sum_\nu\L_{X_\nu}\bar f_\nu
\]
and therefore
\[
 \langle A,g\rangle
 =-\sum_\nu\langle\L_{X_\nu}A,\bar f_\nu\rangle=0.
\]
As usual, subtracting a fixed test function of integral one shows that
$A|_{U_x}=C_x\mu_0|_{U_x}$.  The constants agree on overlaps.  Since
$m\ge2$, removing the finite zero-dimensional stratum $\p^mM$ does not disconnect a connected manifold with corners, and consequently there is one constant $C$ such that
\[
 A=C\mu_0\qquad\text{on }M\setminus\p^mM.
\]
Thus $B:=A-C\mu_0$ is supported on the finite set $\p^mM$.

It remains only to classify this vertex-supported residual; no version of the preceding
claim at $p=m$ is needed.  Fix $v\in\p^mM$ and adapted coordinates
$u^1,\dots,u^m\ge0$ centered at $v$.  A distribution supported at $v$ is a finite sum
$\sum_\gamma b_\gamma\partial^\gamma\delta_v$.  The exact linear fields
\[
 X_{ij}=u^i\partial_{u^i}-u^j\partial_{u^j}\qquad(i\ne j)
\]
show, by taking Lie derivatives, that a coefficient $b_\gamma$ can be nonzero only when
$\gamma_1=\cdots=\gamma_m=:r$.  If $r>0$, use the divergence-free corner field
\[
 Y=(u^1)^2\partial_{u^1}-2u^1u^2\partial_{u^2}.
\]
After multiplying a relative primitive by a cut-off, $Y$ may be realized near $v$ by an element of
$\X_{\mathrm{exact}}(M,\p M,\mu_0)$.  If $T_r(f)=\partial_1^r\cdots\partial_m^r f(0)$, the Leibniz rule gives
\[
 T_r(Yf)=-r(r+1)\,\partial_1^{r-1}\partial_2^r\cdots\partial_m^r f(0).
\]
Thus $\mathcal L_Y(\partial_1^r\cdots\partial_m^r\delta_v)$ is nonzero for every $r>0$, and invariance forces $r=0$.  Therefore $B|_v=c_v\delta_v$, and summing over the finitely many vertices proves the lemma.
\end{proof}

\begin{remark*}
The first part of the proof is intentionally restricted to $p<m$.  At a vertex there is no tangential coordinate $u^{p+1}$, and the relative-form decomposition used there can fail; for example, when $p=m=2$,
\[
 \Omega^1(U_x,\p U_x)\ni f_1u^1du^2+f_2u^2du^1,
 \qquad
 \Omega^0(U_x,\p U_x)=u^1u^2C^\infty(U_x).
\]
The vertex contribution is instead handled by the separate isotropy argument in the second part of the proof.
\end{remark*}

\begin{lemma}\label{old5} 
Each operator 
\begin{align*}
\check G_c:\;& C^\infty(M)\to \mathcal D'(M^{n-1})
\\&
f_i\mapsto \big((f_1,\dots\widehat{f_i}\dots, f_n)\mapsto G_c(f_1,\dots,f_n) \big)
\end{align*}
has the following property: If for $f\in C^\infty(M)$ and a connected open $U\subseteq M$ the 
restriction $f|U$ is constant, then $\L_{X^{\text{diag}}} (\check G_c(f))|U^{n-1} = 0$ for each exact vector field 
$X\in \X_{\text{exact}}(M,\p M,\mu_0)$.
\end{lemma}
 
\begin{proof} Without loss, let $i=1$.
For $x\in U$,
choose $g\in C^\infty(M)$ with $g=1$ near 
$M\setminus U$ and $g=0$ on a large open neighborhood $V$ of $x$. 
Then for any $X \in \X_{\text{exact}}(M,\p M,\mu_0)$, that is $X=\hat\io_{\mu_0}\i(d\om)$ for some 
$\om \in \Om^{m-2}_c(W,\p W)$ where $W\subset M$ is an oriented open set, let $Y = \hat\io_{\mu_0}\i(d(g\om))$. 
The vector field $Y\in \X_{\text{exact}}(M,\p M,\mu_0)$ equals $X$ near $M\setminus U$ and vanishes on $V$. Since $f_1$ is 
constant on $U$,  $\L_X f_1 = \L_{Y} f_1$. 
For all $f_2,\dots,f_n \in C^\infty(M)$ and each $Z\in\X(M,\p M,\mu_0)$ we have
\begin{align*}
\big\langle \L_{Z^{\text{diag}}} \check G_c(f_1)&, f_2\otimes\dots\otimes f_n \big\rangle 
= \left\langle \check G_c(f_1), -\L_{Z^{\text{diag}}} (f_2\otimes\dots\otimes f_n) \right\rangle
\\&
= -  G_c(f_1, \L_Z f_2,\dots, f_n)) - \dots -  G_c(f_1, f_2,\dots, \L_Z f_n)) 
\\&
= G_c(\L_Z f_1, f_2,\dots,f_n)
= \left\langle \check G_c(\L_Z f_1),  f_2\otimes\dots\otimes f_n\right\rangle,
\end{align*}
since $G_c$ is invariant. Thus also
\[
\L_{X^{\text{diag}}} \check G_c(f_1) = \check G_c(\L_X f_1) = \check G_c(\L_{Y} f_1) =
\L_{Y^{\text{diag}}}\check G_c(f_1)\,.
\]
Now $Y$ vanishes on $V$ and so does $\L_{X^{\text{diag}}} (\check G_c(f))$ on $V^{n-1}$. 
This holds also on $U^{n-1}$: Take $x_1,\dots,x_{n-1}\in U$ and choose $V$ in such a way that all $x_i\in V$.
\end{proof}
 
\begin{lemma}\label{old6}
Put $V:=\p^mM$ and let $\lambda_0:=\mu_0$ and $\lambda_v:=\delta_v$ for $v\in V$.
There is a finite linear combination
\[
 P=\sum_{a_1,\dots,a_n\in\{0\}\cup V}
 C_{a_1\dots a_n}\,
 \lambda_{a_1}\otimes\cdots\otimes\lambda_{a_n}
\]
of decomposable invariant distributions such that
$\hat G_c-P$ is supported on the union of the partial diagonals
\[
D:=\{(x_1,\dots,x_n)\in M^n:\ x_i=x_j\text{ for some }i\ne j\}.
\]
Thus, modulo products of invariant one-tensors, the kernel is supported on $D$.
\end{lemma}

\begin{proof}
Let $(x_1,\dots,x_n)\in M^n\setminus D$.  Choose connected good neighborhoods
$U_{x_i}$ with pairwise disjoint closures and test functions
$f_i$ supported in $U_{x_i}$.  Repeating Lemma \ref{old5} as before shows that, after
$f_1,\dots,f_{n-1}$ have been fixed, the distribution
$\check G_c(f_1,\dots,f_{n-1})$ is invariant on a connected open neighborhood containing
$\operatorname{supp}f_n$.  Lemma \ref{old4} therefore gives
\[
 G_c(f_1,\dots,f_n)
 =C_n(f_1,\dots,f_{n-1})\int_Mf_n\mu_0
 +\sum_{v\in V\cap U_{x_n}}
   C_{n,v}(f_1,\dots,f_{n-1})f_n(v).
\]
The coefficient functionals $C_n$ and $C_{n,v}$ are bounded invariant multilinear
functionals of the remaining arguments.  Applying the same argument successively to
$f_{n-1},\dots,f_1$ expresses $G_c$ on such mutually disjoint supports as a finite linear
combination of products
\[
 \lambda_{a_1}(f_1)\cdots\lambda_{a_n}(f_n),
 \qquad a_i\in\{0\}\cup V.
\]
(The coefficient of a term involving the same vertex in two different slots is irrelevant
on $M^n\setminus D$, but it may harmlessly be included in $P$.)  The resulting constants
are independent of the chosen disjoint neighborhoods by overlap and connectedness on each
component of the configuration space; equivalently they are the coefficients of the
corresponding decomposable invariant distributions.

Subtract the resulting $P$.  The distribution $\hat G_c-P$ vanishes on every tensor product
$f_1\otimes\cdots\otimes f_n$ with pairwise disjoint supports.  Such tensor products span a
dense subspace of $C_c^\infty(M^n\setminus D)$, hence
$\operatorname{supp}(\hat G_c-P)\subset D$.
\end{proof}

\subsection{Further reduction toward Main Theorem \ref{maintheorem}}\label{old7} 
By Lemma \ref{old6} we first subtract the full decomposable invariant distribution $P$ occurring there and replace $\hat G_c$ by $\hat G_c-P$.  Thus, without loss of generality, the remaining kernel is supported on the union of the partial diagonals.
Let $(U,u)$ be a good chart on $M$ with coordinates $u^1\ge 0, \dots, u^k\ge0, u^{k+1},\dots, u^m$ 
such that $\mu_0|U = du^1\wedge \dots\wedge du^m$. By \ref{old6}, the distribution $\hat G_c|_{U^n}\in \mathcal D'(U^n)$ has support contained in the union $D(U)$ of all partial  diagonals 
and is of finite order $k$ since $M$ is compact.  

Recall from \ref{old8def} the space 
\begin{align*}
N_r(M,\mu_0) &= \big\{L:C^{\infty}(M)^r \to C^{\infty}(M)\,:\, \L_XL = 0 \text{ for all  } X\in\X(M,\mu_0) \big\}
\end{align*}
of all $r$-linear $\mu_0$-natural operators on $M$.

\begin{proposition}\label{old8prop}
Let $L\in\mc N_r(M,\mu_0)$  have order at most $N$ in each argument. In a volume chart
$\mu_0=du^1\wedge\dots\wedge du^m$ write
\[
 L(f_1,\dots,f_r)=\sum_\al C_\al(u)\,
 \p^{\al_1}f_1\cdots\p^{\al_r}f_r.
\]
Then:
\begin{enumerate}
\item[(a)] The coefficients $C_\al$ are constant. Hence $L$ has a polynomial symbol
\[
 \si(L)(\xi_1,\dots,\xi_r)=\sum_\al C_\al\,
 \xi_1^{\al_1}\cdots\xi_r^{\al_r}.
\]
If $q=\sum_{|\be|=d}q^i_\be u^\be\p_{u^i}$ is a homogeneous polynomial vector field of degree $d\ge1$ with $\on{div}q=0$, naturality is equivalent to
\begin{equation}\label{eq:symbolcond}
\sum_{j=1}^r\sum_{i,\be} q^i_\be\;\xi_{j,i}\;
\p_{\xi_j}^{\be}\,c(\xi_1,\dots,\xi_r)=0.
\end{equation}
Thus degree $d=0$ in the jet of a divergence-free vector field gives constant coefficients, while the conditions $d\ge1$ act on the symbol.

\item[(b)] For $d=1$, \eqref{eq:symbolcond} is precisely infinitesimal
$\on{SL}(m,\mathbb R)$-invariance of $c$ under the diagonal action on the variables $\xi_j$. In the component of multidegree $(k_1,\dots,k_r)$ put $K=\sum_jk_j$. This component is zero unless $K=m\ell$ for some $\ell\ge0$. If $K=m\ell$, it is
\[
 \left(\bigotimes_{j=1}^r\on{Sym}^{k_j}(\mathbb R^m)^*\right)^{\on{SL}(m)}
 \cong
 \on{Hom}_{\on{GL}(m)}\!\left((\det)^\ell,
       \bigotimes_{j=1}^r\on{Sym}^{k_j}(\mathbb R^m)^*\right),
\]
Equivalently, by the First Fundamental Theorem and straightening, it is spanned by products of brackets
\[
 [\xi_{a_1},\dots,\xi_{a_m}]
   =\det(\xi_{a_1},\dots,\xi_{a_m}),
\]
and has a basis indexed by semistandard Young tableaux of rectangular shape $(\ell^m)$ and content $(k_1,\dots,k_r)$. A column is one determinant bracket; the straightening (Pluecker) relations are exactly the relations between arbitrary bracket monomials. In particular its dimension is the corresponding Kostka/Littlewood--Richardson multiplicity. After full polarization this may equivalently be described by standard tableaux.  The first rectangular shapes are, for $m=2$,
\[
 \ydiagram{1,1}\;,\qquad \ydiagram{2,2}\;,\qquad \ydiagram{3,3}\;,\dotsc
\]
and, for $m=3$,
\[
 \ydiagram{1,1,1}\;,\qquad \ydiagram{2,2,2}\;,\dotsc
\]
The one-column diagram 
with $m$ boxes describes one Jacobian
bracket; juxtaposed columns encode products of brackets before straightening.

\item[(c)] The equations \eqref{eq:symbolcond} for $d\ge2$ are the genuinely nonlinear volume-preserving naturality conditions. Thus
\[
 \mc N_r(M,\mu_0)
 =\bigcap_{d\ge2}\ker D_d
\]
inside the direct sum of the rectangular-tableau spaces from (b), where $D_d$ denotes the linear operator on symbols defined by the left hand side of \eqref{eq:symbolcond} for homogeneous divergence-free fields of degree $d$. Only finitely many $d$ occur for a differential operator of fixed order.

\item[(d)] The Jacobians
$J_S(f_{j_1},\dots,f_{j_m})=df_{j_1}\wedge\dots\wedge df_{j_m}/\mu_0$
are natural, and products and insertions $L_1(\dots,L_2(\dots),\dots)$ of natural operators are natural.
\end{enumerate}
\end{proposition}

\begin{proof}
(a) First work in an interior volume chart. The constant vector fields $\p_{u^i}$ preserve $\mu_0$, so $\L_{\p_{u^i}}L= 0$ gives
\[
 [\p_{u^i},L](f_1,\dots,f_r)=0.
\]
Comparing coefficients of the independent jets of the $f_j$ gives
$\p_{u^i}C_\al=0$ for every $i$ and $\al$. Thus the coefficients are constant. In a boundary volume chart the same conclusion follows on its connected interior part and extends smoothly to the boundary. Consequently constant coefficients need not be included in Definition \ref{old8def}; they follow from naturality.

Now fix a point and translate it to $0$ in a volume chart. Once the degree-zero part has been used as above, the Taylor expansion of a divergence-free vector field has divergence-free homogeneous components $q_d$, $d\ge1$. Take $f_j=e^{\xi_j\cdot u}$. Then
$L(f_1,\dots,f_r)=\si(L)(\xi)e^{\sum\xi_j u}$ and
\[
 qf_j=\sum_{i,\be}q^i_\be\,\xi_{j,i}\,u^\be e^{\xi_j u}
 =\sum_{i,\be}q^i_\be\,\xi_{j,i}\,
   \p_{\xi_j}^{\be}e^{\xi_j u}.
\]
Evaluating \eqref{eq:equiv} at $0$ gives \eqref{eq:symbolcond}. Conversely, exponentials separate the constant coefficients, so these identities imply the infinitesimal equivariance relation.

(b) For $d=1$, write $q(u)=Bu$ with $B\in\mf{sl}(m,\mathbb R)$. Equation \eqref{eq:symbolcond} is the infinitesimal $\on{SL}(m)$-invariance of the symbol. The First Fundamental Theorem \cite[Ch.~II]{Weyl46} gives the determinant brackets. For Young tableaux, Schur functors, and the straightening calculus we use \cite{Fulton97}; for standard monomial theory and its flag-variety form see also \cite{LakRag08}. For fixed multidegree, the usual Schur--Weyl/Pieri decomposition shows that an $\on{SL}(m)$-fixed vector can occur only in the rectangular highest weight $(\ell^m)$, and the straightening theorem identifies the bracket basis with semistandard tableaux of that rectangular shape and the prescribed content.

(c) is just the remaining part, $d\ge2$, of \eqref{eq:symbolcond}; for fixed differential order derivatives $\p_{\xi_j}^{\be}$ of sufficiently high degree vanish.

(d) If $\L_X\mu_0=0$ then
$\L_X(df_1\wedge\dots\wedge df_m)=\sum_jdf_1\wedge\dots\wedge d(Xf_j)\wedge\dots\wedge df_m$, hence
$XJ_S(f)=\sum_jJ_S(\dots,Xf_j,\dots)$. Since $X$ is a derivation, the same identity is preserved under products and insertions.
\end{proof}

\begin{lemma}[Ambient scalarization principle]\label{ambientscalarization}
Let $V$ be an oriented $m$-dimensional vector space with volume form $\mu$.  There is no
$\on{SL}(V,\mu)$-equivariant contraction $\Lambda^qV^*\to\mathbb R$ for
$0<q<m$.  More generally, in the tensor algebra generated by covariant jets of scalar
functions and $\mu^{\pm1}$, every scalar contraction using the alternating tensor supplied
by $\mu$ consumes exactly $m$ free covariant indices.  Thus an alternating block of fewer
than $m$ covariant indices is intrinsically a differential form; it cannot be made into a
scalar by ``division by $\mu$''.

Consequently, if $F_\mu$ is an ambient $\mu$-natural scalar differential operator, then
$F_\mu|_S$ is a natural scalar coefficient on every stratum $S$, whereas a partial
Jacobian made from fewer than $m$ differentials is not such a scalar.  A block of
$q=m-p$ differentials has instead the correct degree to be pulled back to a codimension
$p$ stratum and integrated there.
\end{lemma}

\begin{proof}
The first assertion is representation theory: $\Lambda^qV^*$ has no trivial
$\on{SL}(V)$-summand for $0<q<m$.  The only alternating invariant tensor of purely
covariant type is, up to scale, $\mu\in\Lambda^mV^*$ (and dually $\mu^{-1}$ in degree
$m$).  Hence contraction with $\mu^{-1}$ removes covariant indices only in complete
blocks of $m$.  Applied to first derivatives this says precisely that
\[
 \frac{df_{i_1}\wedge\cdots\wedge df_{i_m}}{\mu}
\]
is an ambient natural scalar, while no analogous quotient exists with fewer than $m$
differentials.  Restriction of an already constructed scalar is of course functorial,
and pullback of a $q$-form to a $q$-dimensional stratum is functorial as well.
\end{proof}

\begin{remark}[Scope of the stronger boundary factorization]\label{scalarizationcaution}
For a positive-dimensional stratum $S=(\p^pM)_j$, $p<m$, the scalarization principle combines
with the relative Spencer--Koszul calculation below to show that, modulo tangential exact forms
and deeper strata, every invariant top-degree boundary functional has ambient natural scalar
coefficients.  This is stronger than mere invariance under the smaller corner isotropy.
At a vertex $p=m$, however, there is no positive-dimensional relative form block and the statement
fails: the linear isotropy is only the diagonal determinant-one torus.  For example, in dimension
$2$ the bilinear operator
\[
 (f,g)\longmapsto (\partial_{u^1}f)(v)(\partial_{u^2}g)(v)
\]
is invariant in adapted volume coordinates at a vertex, but is not the restriction of an ambient
$\on{SL}(2)$-natural scalar operator.  Hence vertices have to be treated separately as in
\eqref{eq:vertexgenerator}.
\end{remark}

\begin{proposition}[Explicit classification of the vertex-natural operators]\label{vertexclassification}
Let $v\in\p^mM$ and choose adapted oriented volume coordinates
$u=(u^1,\ldots,u^m)$ centered at $v$, so that the incident local faces are
$\{u^i=0\}$ and $\mu_0=du^1\wedge\cdots\wedge du^m$.  We temporarily order the incident
faces in order to label the coordinates; changing that ordering simultaneously permutes the
coordinate indices in all equations below, so the resulting invariant space is intrinsic.
For an $n$-linear scalar differential operator at $v$ write
\begin{multline}\label{eq:vertexsymbol}
 \mathcal V(f_1,\ldots,f_n)
 =\sum_{\alpha_1,\ldots,\alpha_n}c_{\alpha_1\ldots\alpha_n}
   \prod_{r=1}^n\partial^{\alpha_r}f_r(v),
 \\
 \text{ with symbol }\quad
 \si(\mathcal V)(\xi_1,\ldots,\xi_n)
 =\sum c_{\alpha_1\ldots\alpha_n}\prod_{r=1}^n\xi_r^{\alpha_r},
\end{multline}
where the sum is finite and $\xi_r=(\xi_{r,1},\ldots,\xi_{r,m})$.
Put
\[
 E_i=\sum_{r=1}^n\xi_{r,i}\partial_{\xi_{r,i}},
 \qquad
 A_i^\gamma=\sum_{r=1}^n
   \xi_{r,i}\partial_{\xi_r}^{\gamma+e_i}
 \quad (\gamma\in\mathbb N^m).
\]
Then $\mathcal V$ is natural for the germ of the identity component of the
face-preserving, $\mu_0$-preserving corner pseudogroup fixing $v$ if and only if
its symbol satisfies
\begin{align}
 (E_i-E_m) \si(\mathcal V)&=0, &&1\le i<m, \label{eq:vertextorus}\\
 (\gamma_k+1)A_i^\gamma \si(\mathcal V)-(\gamma_i+1)A_k^\gamma  \si(\mathcal V)&=0,
 &&|\gamma|\ge1,\quad 1\le i,k\le m. \label{eq:vertexhigher}
\end{align}
For a symbol of finite order only finitely many equations
\eqref{eq:vertexhigher} are nonzero.  Thus the space of vertex-natural operators
of any prescribed order and arity is the kernel of an explicit finite matrix.

Equivalently, \eqref{eq:vertextorus} says that each monomial occurring in $ \si(\mathcal V)$
has a common total coordinate weight
\begin{equation}\label{eq:vertexweight}
 \sum_{r=1}^n\alpha_{r,1}=\cdots=
 \sum_{r=1}^n\alpha_{r,m}=:\ell.
\end{equation}
The integer $\ell\ge0$ gives a direct-sum grading of the vertex-natural space.
For $\ell=0$ the space is one-dimensional, generated by
$\prod_r f_r(v)$.  For $\ell=1$ it has the basis
\begin{equation}\label{eq:vertexweightone}
 \left(\prod_{r\notin\operatorname{im}\rh}f_r(v)\right)
 \prod_{i=1}^m\partial_{u^i}f_{\rh(i)}(v),
 \qquad
 \rh:\{1,\ldots,m\}\hookrightarrow\{1,\ldots,n\},
\end{equation}
so it is zero for $n<m$ and has dimension $n!/(n-m)!$ for $n\ge m$.
\end{proposition}

\begin{proof}
The finite jet group at $v$ of the face-preserving, volume-preserving corner
pseudogroup has connected linear part
\[
 T=\{\operatorname{diag}(a_1,\ldots,a_m):a_i>0,\ \prod_i a_i=1\}
\]
and a connected unipotent kernel.  The infinitesimal linear fields are
$\sum_i b_i u^i\partial_{u^i}$ with $\sum_i b_i=0$.  Their action on the symbol
\eqref{eq:vertexsymbol} is $\sum_i b_iE_i$, hence invariance under $T$ is
exactly \eqref{eq:vertextorus}, equivalently \eqref{eq:vertexweight}.

For $d\ge2$, a homogeneous degree-$d$ infinitesimal corner automorphism has the
form
\[
 X_d=\sum_{i=1}^m\sum_{\substack{|\beta|=d\\ \beta_i\ge1}}
 q^i_\beta u^\beta\partial_{u^i}.
\]
Tangency to every incident face is exactly the condition $\beta_i\ge1$ in the
$i$-th component.  The divergence-free equation is, coefficient by coefficient,
\begin{equation}\label{eq:vertexdiv}
 \sum_{i=1}^m(\gamma_i+1)q^i_{\gamma+e_i}=0,
 \qquad |\gamma|=d-1.
\end{equation}
The infinitesimal action on a constant-coefficient symbol is
\[
 D_{X_d} \si(\mathcal V)=
 \sum_{r=1}^n\sum_{i=1}^m
 \sum_{\substack{|\beta|=d\\\beta_i\ge1}}
 q^i_\beta\,\xi_{r,i}\partial_{\xi_r}^{\beta} \si(\mathcal V).
\]
For a fixed $\gamma$ set $B_i=A_i^\gamma  \si(\mathcal V)$.  By \eqref{eq:vertexdiv}, the vector
$q_i=q^i_{\gamma+e_i}$ ranges over the hyperplane
$\sum_i(\gamma_i+1)q_i=0$.  Therefore
$\sum_iq_iB_i=0$ for every admissible $q$ if and only if
$B_i/(\gamma_i+1)$ is independent of $i$, which is precisely
\eqref{eq:vertexhigher}.  The image of the identity component of the corner pseudogroup in every finite jet
group is connected; therefore invariance under its Lie algebra is equivalent to invariance
under that image.  This proves necessity and sufficiency.

For the last assertion, suppose $\ell=1$.  Every exponent
$\alpha_{r,i}$ is then either $0$ or $1$, and the supports of the nonzero
$\alpha_r$ form a set partition of the coordinate labels $\{1,\ldots,m\}$.
If one block contains two distinct labels $i,k$, use
\eqref{eq:vertexhigher} with $\gamma=e_k$ and the pair $(i,k)$.
The term obtained by differentiating that block with
$\partial_{\xi_r}^{e_i+e_k}$ cannot be cancelled by an $A_k^{e_k}$ term,
because no monomial of weight $\ell=1$ contains $\xi_{r,k}^2$; nor can acting
on a different slot give the same resulting monomial.  Its coefficient must
therefore vanish.  Hence every nonempty block is a singleton.  Conversely,
once every differentiated slot contains exactly one first derivative, all
operators $A_i^\gamma$ with $|\gamma|\ge1$ vanish.  This gives exactly the
injections \eqref{eq:vertexweightone}.
\end{proof}

\begin{remark}[The first genuinely higher-jet vertex block]\label{vertexexample22}
For $m=2$, weight $\ell=2$ and arity $n=3$, 
a direct computation from \(\eqref{eq:vertexhigher}\) gives a six-dimensional space.  One element is
\begin{equation}\label{eq:vertexQ}
 Q(f,g,h)=
 f_2g_2h_{11}+f_2g_{12}h_1+f_{22}g_1h_1,
\end{equation}
where $f_i=\partial_{u^i}f(v)$ and
$f_{ij}=\partial_{u^i}\partial_{u^j}f(v)$, and similarly for $g,h$.
The six operators obtained from $Q$ by permuting the three arguments form a
basis.  In the same weight the spaces of arity $1$ and $2$ are zero.  Thus
higher derivatives genuinely occur at a vertex; the complete classification is
the kernel system \eqref{eq:vertextorus}--\eqref{eq:vertexhigher}, rather than a
first-jet generator statement.
\end{remark}

\begin{proposition}[Parabolic tableaux for the relative boundary form]\label{boundarytableaux}
Let $S=(\p^pM)_j$, $q=m-p$, and choose adapted volume coordinates at $x\in S$.
The linear isotropy of the identity component of the corner diffeomorphism group fixing $x$ preserves each of the $p$ local boundary hypersurfaces.  Thus on tangent vectors it is
\[
 g=\begin{pmatrix}A&0\\ C&D\end{pmatrix},\qquad
 A=\operatorname{diag}(a_1,\dots,a_p),\quad a_i>0,\qquad
 (a_1\cdots a_p)\det D=1.
\]
This is the group-level counterpart of the infinitesimal condition in Subsection \ref{old3}.
For scalar coefficients which are restrictions of ambient natural operators, the
linear symbols are therefore the same rectangular $\on{SL}(m)$ tableaux as in
Proposition \ref{old8prop}.  The genuinely parabolic representation occurs in the
\emph{relative $q$-form factor}: its pullback to $S$ transforms by the top-form
character of $D$.  After polarization, flag-minor straightening organizes these
relative symbols by columns of heights $q,q+1,\ldots,m$; the $q$ free horizontal
indices are retained as form indices rather than divided by $\mu$.

In particular, full columns of height $m$ give ambient determinant brackets and hence
ambient scalar factors, while the remaining relative character is represented by the
$q$-form block.  The smallest such block is
\[
 df_{i_1}|_S\wedge\cdots\wedge df_{i_q}|_S,
\]
represented by a single column $\ydiagram{1,1}$ when $q=2$, by
$\ydiagram{1,1,1}$ when $q=3$, and so on.  Products are straightened by the usual
flag-tableau laws \cite{LakRag08,Seshadri16}.
\end{proposition}

\begin{remark}[Higher jets]\label{boundaryhigherjets}
Higher corner jets still give useful infinitesimal equations for a relative $q$-form
symbol, but they should be imposed on the whole form-valued symbol, not used to create a
new class of boundary scalar operators.  After the full-height columns have been
identified with ambient natural scalar symbols, tangential integration by parts moves
horizontal derivatives between factors; modulo exact $q$-forms, the remaining relative
part is represented by $q$ exterior differentials of input functions.  This is the
form-valued version of the straightening problem used in Proposition
\ref{boundaryfactorization}.
\end{remark}

\begin{lemma}[Relative Spencer--Koszul lemma]\label{horizontalibp}
Let $S=(\p^pM)_j$ have positive dimension $q=m-p>0$.  Let $\mathcal C_S^q(n)$ be the space of
multilinear local differential operators $C^\infty(M)^n\to\Omega^q(S)$ obtained from ambient jets,
and quotient it by $d_S\mathcal C_S^{q-1}(n)$.  Invariant classes for the face-preserving,
$\mu_0$-preserving corner pseudogroup are represented by finite sums
\begin{equation}\label{eq:relativeSKnormalform}
 \Big(\prod_{i\in S_0}f_i\Big)
 F_{\mu_0}(f_s:s\in T)\,
 d_Sf_{i_1}\wedge\cdots\wedge d_Sf_{i_q},
 \quad
 S_0\sqcup T\sqcup\{i_1,\ldots,i_q\}=\{1,\ldots,n\},
\end{equation}
with $F_{\mu_0}$ an ambient $\mu_0$-natural scalar multilinear differential operator.
Equivalently, the natural map from the span of the expressions \eqref{eq:relativeSKnormalform}
to the invariant part of the quotient $\mathcal C_S^q(n)/d_S\mathcal C_S^{q-1}(n)$ is surjective.
\end{lemma}

\begin{proof}
We use the filtration by total jet order.  Work at a point of $S$ in adapted volume coordinates
$(u^1,\ldots,u^p,w^1,\ldots,w^q)$ and write $V^*=N^*\oplus W^*$, with
$W^*=\operatorname{span}\{dw^1,\ldots,dw^q\}$.  For the principal symbol of an $n$-linear
operator write $\xi_r=(\nu_r,\tau_r)\in N^*\oplus W^*$, $1\le r\le n$.
At the associated-graded level the horizontal differential is the Koszul differential for the
$q$ components of the total tangential covector
\[
 T=\tau_1+\cdots+\tau_n.
\]
Indeed, the principal symbol of a total tangential derivative is multiplication by the
corresponding component of $T$.  Hence top horizontal symbols modulo tangential integration by
parts are polynomial symbols modulo the regular sequence $(T_1,\ldots,T_q)$.  This is the
symbolic form of
\[
 A\,d_Sh\wedge\eta\equiv_S-h\,d_SA\wedge\eta-hA\,d_S\eta,
\]
and it makes the termination of integration by parts algebraic rather than dependent on a
choice of an ``excess order''.

Let $H_{p,q}$ be the linear corner isotropy
\[
 \begin{pmatrix}A&0\\ C&D\end{pmatrix},\qquad
 A=\operatorname{diag}(a_1,\ldots,a_p),\quad
 (a_1\cdots a_p)\det D=1.
\]
A top $q$-form symbol is an $H_{p,q}$-semi-invariant of character $\chi(g)=\det D$.
After polarization, invariance under the unipotent block $C$ and under $\on{SL}(W)$ is described
by the usual flag minors.  Standard-monomial straightening gives a basis by products of nested
flag columns.  A column containing the $q$ tangential rows and a subset $I$ of the normal rows
has torus weight
\[
 \det D\prod_{i\in I}a_i=\prod_{i\notin I}a_i^{-1}.
\]
In a standard monomial the omitted-normal sets are nested.  Since the total character is
$\chi=\prod_{i=1}^pa_i^{-1}$, every normal index has to be omitted exactly once.  Nestedness
therefore forces exactly one column to omit all normal rows, hence one column of height $q$,
and every remaining column has height $m$.  Thus every associated-graded relative symbol is a
sum of one tangential $q$-minor times products of full $m$-minors.

We next use the Koszul relations $T_a=0$.  If a label occurring in the $q$-column occurs also in
a full column, replace its tangential vector by minus the sum of the other tangential vectors.
Terms in which a label already present in the $q$-column is inserted vanish by alternation; the
remaining terms are straightened again.  With the standard tableau order this is triangular and
strictly decreases the multiplicity, outside the $q$-column, of the ordered labels occurring in
that column.  Repetition terminates.  We obtain representatives in which the $q$ labels of the
short column occur nowhere in the full columns.  Labels occurring in no column at all give the
undifferentiated factors $f_i$.  The full columns involve only the complementary set $T$ and,
by the first fundamental theorem for $\on{SL}(V)$, their products are precisely the scalar
$\on{SL}(V)$-invariant symbols in those variables.  This proves the associated-graded version of
\eqref{eq:relativeSKnormalform}.

It remains to impose the higher-jet equations.  For $d\ge2$ let
\[
 \mathfrak g_d=\ker\bigl(\operatorname{div}:S^dV^*\otimes V\longrightarrow S^{d-1}V^*\bigr)
\]
be the space of homogeneous degree-$d$ infinitesimal volume-preserving coordinate changes.
As an $\on{SL}(V)$-module, $\mathfrak g_d$ is the irreducible Cartan summand of
$S^dV^*\otimes V$ complementary to the divergence image.  The subspace consisting of jets of
corner-preserving fields is nonzero: if $q>0$ one may take a tangential field with coefficient a
pure normal monomial, and if necessary add the usual divergence correction.  Hence its
$\on{SL}(V)$-span is all of $\mathfrak g_d$.

For $d\ge2$ the first jet of such a homogeneous vector field vanishes at the base point, so its
action on the short column
$d_Sf_{i_1}\wedge\cdots\wedge d_Sf_{i_q}$ is zero at principal-symbol level.  The higher-jet
naturality equations therefore act only on the scalar full-column factor.  That factor is already
$\on{SL}(V)$-invariant.  If it is annihilated by the corner-preserving subspace of
$\mathfrak g_d$, $\on{SL}(V)$-invariance implies that it is annihilated by the whole
$\on{SL}(V)$-span, hence by all of $\mathfrak g_d$.  Thus it satisfies exactly the ambient
higher-jet equations of Proposition \ref{old8prop} and is the symbol of an ambient
$\mu_0$-natural scalar operator.

We can now argue downwards in the finite jet-order filtration.  Subtract from the given invariant
class an expression of the form \eqref{eq:relativeSKnormalform} with the same highest associated-
graded symbol.  The remainder has strictly smaller total jet order.  Iteration terminates and
proves the lemma.  Notice that no normal integration by parts is used.  Finally, an exact
$q$-form changes the integral over $S$ by Stokes only on $\partial\overline S$, hence by a
functional supported on deeper strata.
\end{proof}

\begin{definition}[Boundary-natural top-form operators]\label{boundarytopform}
Let $S=(\p^pM)_j$ and $q=m-p$.  A multilinear local differential operator
\[
 \mathcal B_{\mu,S}:C^\infty(M)^n\longrightarrow\Omega^q(S)
\]
is called boundary-natural if its value at $x\in S$ depends on a finite ambient jet of
$(\mu,f_1,\ldots,f_n)$ at $x$ and, for every local corner diffeomorphism $\ph$ preserving
$S$ and the volume form,
\[
 \mathcal B_{\ph^*\mu,S}(\ph^*f_1,\ldots,\ph^*f_n)
 =\ph|_S^*\mathcal B_{\mu,S}(f_1,\ldots,f_n).
\]
Two such operators are identified modulo $d_S$-exact top forms when the resulting
functional is considered modulo deeper strata.
\end{definition}

\begin{proposition}[Strong boundary factorization]\label{boundaryfactorization}
Let $S=(\p^pM)_j$ and $q=m-p$.

If $q>0$, every local multilinear differential functional supported on the full diagonal over $S$
and invariant under the volume-preserving corner pseudogroup is, modulo a functional supported on
deeper strata, a finite sum of
\begin{equation}\label{eq:boundaryfactorization}
 (f_1,\ldots,f_n)\longmapsto
 \int_S
 \Big(\prod_{i\in S_0}f_i\Big)
 F_{\mu_0}(f_s:s\in T)\,
 d_Sf_{i_1}\wedge\cdots\wedge d_Sf_{i_q},
\end{equation}
where
$S_0\sqcup T\sqcup\{i_1,\ldots,i_q\}=\{1,\ldots,n\}$ and $F_{\mu_0}$ is an ambient
$\mu_0$-natural multilinear scalar differential operator.  Conversely every functional of
\eqref{eq:boundaryfactorization} is invariant.

If $q=0$, the corresponding statement with only ambient natural scalar operators is false in
general.  The full-diagonal contribution at the vertex is an arbitrary vertex-natural scalar
multilinear differential operator in the sense of \eqref{eq:vertexgenerator}.
\end{proposition}

\begin{proof}
Choose an adapted volume chart
$(u',w)\in[0,\infty)^p\times\mathbb R^q$ with $S=\{u'=0\}$.  The structure theorem for
distributions supported on a submanifold writes a full-diagonal functional locally as a finite-order
differential functional along $S$.  Tangential translations make its coefficients constant in
adapted volume coordinates at the symbol level, and invariance under changes of adapted coordinates
makes the local expression a boundary-natural horizontal top-form operator in the sense of
Definition \ref{boundarytopform}.  Changing its representative by a $d_S$-exact form changes the
integral only on $\partial\overline S$, hence only on deeper strata.

For $q>0$, Lemma \ref{horizontalibp} applies to the class of this top-form operator modulo exact
forms and gives precisely the finite sum \eqref{eq:boundaryfactorization}.  Conversely, ambient
naturality of $F_{\mu_0}$, functoriality of $d$, pullback to $S$, and functoriality of integration
of top forms prove invariance of every displayed functional.

For $q=0$, there is no horizontal Koszul complex and no short relative column.  The linear isotropy
at a vertex is only the diagonal determinant-one torus, so additional scalar invariants survive.
For example, in dimension $2$,
\[
 (f,g)\longmapsto (\partial_{u^1}f)(v)(\partial_{u^2}g)(v)
\]
is unchanged under every face-preserving volume coordinate change at the level of first jets, and
higher jets of the coordinate change do not alter first derivatives at the fixed vertex.  It is not
$\on{SL}(2)$-invariant (its antisymmetrization is the ambient Jacobian).  Hence the correct vertex
class is the vertex-natural class stated in the theorem.
\end{proof}

\begin{lemma}\label{old8}
Let $\hat G\in\mc D'(M^n)$ be a $\Diff_0(M,\mu_0)$-invariant distribution supported on the
full diagonal.  Then
\[
G(f_1,\dots,f_n)=\int_M L(f_1,\dots,f_{n-1})f_n\,\mu_0+G_\p(f_1,\dots,f_n),
\]
where $L\in\mc N_{n-1}(M,\mu_0)$ is natural and $G_\p$ is supported in $\De(\p M)$.
Moreover, for every positive-dimensional stratum $S=(\p^pM)_j$, $p<m$, its restriction
modulo terms supported on deeper strata is a finite sum of the functionals
\eqref{eq:boundaryfactorization}; at a vertex it is a finite sum of vertex-natural scalar
operators \eqref{eq:vertexgenerator}.  Thus the full-diagonal boundary terms are exactly of
the form used in the Main Theorem.
\end{lemma}

\begin{proof}
Let $(U, u)$ be a good open chart on $M$, diffeomorphic to  $Q^m_p$ with coordinates
$u^1\ge 0, \dots, u^p\ge 0, u^{p+1}, \dots, u^m$, such that $\mu_0|U = du^1 \wedge \dots \wedge du^m$. The distribution $\hat G|_U \in \mc D'(U^n)$ has support contained in the full diagonal 
$\De(U)=\{(x,\dots,x) \in U^n \,:\, x \in U\}$ and is of finite order $k$ since $M$ is compact. By \cite[Thm.~2.3.5]{Hor1983}, the corresponding multilinear form $G$ can be written as
\[
G(f_1, \dots, f_n) = \sum_{|\al_1| + \ldots + |\al_{n-1}| \leq k} 
\left\langle A_{\al_1, \dots, \al_{n-1}}, \p^{\al_1} f_1\ldots\p^{\al_{n-1}} f_{n-1}. f_n \right\rangle\,,
\]
with multi-indices $\al_j = (\al_{j,1}, \dots, \al_{j,m})$ and distributions $A_{\al_1, \dots, \al_{n-1}} \in \mc D'(U)$ of order $k-|\al_1| - \ldots -|\al_{n-1}|$. Furthermore the distributions in this representation are uniquely determined. We shall write $\al = (\al_1, \dots, \al_{n-1})$ and $|\al| = |\al_1| + \ldots +|\al_{n-1}|$. The special role of $f_n$ here is without loss of generality; any $f_i$ can be taken instead, by integration by parts. 
 
For $x \in \p^p U$ we choose a good set $U_x$ with $x \in U_x \subset \overline{U_x} \subset U$ and $X_{ij}, X_k \in \mf X_{\mathrm{exact}}(M, \p M, \mu_0)$  with
\begin{align*}
X_{ij}|_{U_x} &= u^i\p_{u^i}- u^j\p_{u^j}
\quad\text{ for } i\le p \text{ and  }i< j \\
X_k|_{U_x} &= \p_{u^k}\quad\text{ for  }k=p+1,\dots,m\,.
\end{align*}
Given functions $f_1, \dots, f_n \in C^\infty_c(U_x)$ and $p+1 \leq k \leq m$ , 
we have by the invariance of $G$,
\begin{align*}
0 &= \sum_{j=1}^n G(f_1, \dots, \L_{X_k}f_j,\dots, f_n) = 
\left\langle \hat G|_{U^n}, \sum_{j=1}^n f_1 \otimes \dots \otimes \L_{X_k}f_j \otimes \dots \otimes f_n \right\rangle
\\&
=\sum_\al \left\langle  A_\al, 
\sum_{j=1}^{n-1} \p^{\al_1}f_1 \ldots \p^{\al_j} \p_{u^k} f_j \ldots \p^{\al_{n-1}} f_{n-1} \, f_n +
\p^{\al_1} f_1\ldots\p^{\al_{n-1}} f_{n-1}\, \p_{u^k}f_n
\right\rangle
\\&
=\sum_\al \left\langle  A_\al, 
\p_{u^k}(\p^{\al_1} f_1\ldots\p^{\al_{n-1}} f_{n-1}. f_n) \right\rangle
=\sum_\al \langle -\p_{u^k} A_\al, \p^{\al_1} f_1\ldots\p^{\al_{n-1}} f_{n-1}. f_n \rangle\,.
\end{align*}
Since the corresponding operator has again a kernel distribution that is supported on the diagonal and since the distributions in the representation are unique, we can conclude that $\p_{u^k} A_\al|_{U_x} = 0$ for each $\al$ and each $p<k\le m$, when $x \in \p^p M$.
 
For the vector fields $X_{ij}$ we use the identity, where $1\le \ell\le n-1$,
\begin{align*}
\p^{\al_\ell} \left( u^i \p_{u^i} - u^j \p_{u^j} \right) f
&= \left( u^i \p_{u^i} \p^{\al_\ell} + \al_{\ell,i} \p^{\al_\ell} - u^j \p_{u^j} \p^{\al_\ell} - \al_{\ell,j} \p^{\al_\ell} \right) f \\
&= \left( u^i \p_{u^i} - u^j \p_{u^j} + \al_{\ell,i} - \al_{\ell,j} \right) \p^{\al_\ell} f \\
&= \left( \p_{u^i} u^i - \p_{u^j} u^j + \al_{\ell,i} - \al_{\ell,j} \right) \p^{\al_\ell} f\,.
\end{align*}
Applied to the distribution $\hat G$ we obtain
\begin{align*}
0 &= \sum_{k=1}^n G(f_1, \dots, \L_{X_{ij}}f_k,\dots, f_n) \\
&= \sum_\al \left\langle  A_\al, 
\Big(\p_{u^i}u^i - \p_{u^j} u^j + \sum_{k=1}^{n-1} (\al_{k,i} - \al_{k,j})\Big) (\p^{\al_1} f_1\ldots\p^{\al_{n-1}} f_{n-1}. f_n) \right\rangle\,.
\end{align*}
Put $S_i=\sum_{k=1}^{n-1}\al_{k,i}$ and $S_j=\sum_{k=1}^{n-1}\al_{k,j}$. Then uniqueness gives the distributional identity
\begin{align*}
0 &= -u^i\p_{u^i} A_\al + u^j\p_{u^j} A_\al +(S_i-S_j)A_\al \\
&= -u^i\p_{u^i} A_\al +(S_i-S_j)A_\al\quad\text{ if }j>p,
\end{align*}
on $U_x$, since we have seen above that $\p_{u^j} A_\al|_{U_x} = 0$ for $j>p$.
 
\smallskip\noindent
\emph{In the interior: Suppose $U\subseteq M \setminus \p M$.} 
To see that $\p_{u^i}A_\al =0$ for all $i$ implies that $A_\al|U_x = C_{\al} \mu_0|U_x$, let 
$f \in C_c^\infty(U_x)$ with 
$\int_M f \mu_0 = 0$. Then there exists $\om \in \Om_c^{m-1}(U_x)$ with 
$d\om = f\mu_0$. In coordinates we have 
$\om = \sum_i \om_i. du^1 \wedge \ldots \widehat{du^i} \ldots \wedge du^m$, 
and so $f = \sum_i (-1)^{i+1} \p_{u^i} \om_i$ with $\om_i \in C_c^\infty(U_x)$. Thus  
\begin{align*}
\langle A_\al, f \rangle &= \sum_i (-1)^{i+1} \langle A_\al, \p_{u^i}\om_i \rangle
= \sum_i (-1)^{i} \langle \p_{u^i} A_\al, \om_i \rangle = 0\,.
\end{align*}
Hence $\langle A_\al, f \rangle = 0$ for all $f \in C_c^\infty(U_x)$ with zero integral and as in 
the proof of \ref{old4} we can conclude that $A_\al|U_x = C_{\al} \mu_0|U_x$.
 
But then, for every $f_n\in C_c^\infty(U_x)$,
\begin{align}
G(f_1,\dots,f_n) &= \int_{U_x} L(f_1,\dots,f_{n-1})\, f_n\, \mu_0\,,\qquad \text{ where }\label{eq:old8.1}
\\
L(f_1,\dots,f_{n-1})&:=\sum_{|\al| \leq k} C_\al\, \p^{\al_1} f_1\dots \p^{\al_{n-1}} f_{n-1}\,. 
\end{align}
 
\smallskip\noindent\emph{Step 1 ($L$ is natural).} Here $U_x\subset M\setminus\p M$. Let $X\in\X_{\mathrm{exact}}(M,\p M,\mu_0)$ have support in $U_x$; then $\on{div}_{\mu_0}X=0$, and by Stokes $\int X(h)\,\mu_0=\int \on{div}(hX)\mu_0=0$ for each $h$. Invariance of $G$ under $X$, together with \eqref{eq:old8.1} and $\int L(f_1,\dots,f_{n-1})\,Xf_n\,\mu_0=-\int X\big(L(f_1,\dots,f_{n-1})\big)f_n\,\mu_0$, gives
\[
0=\int_{U_x}\Big(\sum_{j=1}^{n-1}L(f_1,\dots,Xf_j,\dots,f_{n-1})-X\big(L(f_1,\dots,f_{n-1})\big)\Big)\,f_n\,\mu_0\,.
\]
As $f_n\in C_c^\infty(U_x)$ is arbitrary, \eqref{eq:equiv} holds on $U_x$ for such $X$ and $f_1,\dots,f_{n-1}\in C_c^\infty(U_x)$, and then for all smooth $f_j$ since $L$ is local. 

Every polynomial divergence free vector field $q$ is of the form $q=\iota^{-1}d\om$ with $\om$ a polynomial $(m-2)$-form (Poincar\'e lemma on a ball, $\iota^{-1}$ as in \ref{old3}); replacing $\om$ by $\chi\om$ with a cut-off $\chi\equiv1$ near a given point $y\in U_x$ gives $X\in\X_{\mathrm{exact}}(M,\p M,\mu_0)$ with support in $U_x$ and $X=q$ near $y$. Hence \eqref{eq:equiv} holds near $y$ for every polynomial divergence free $q$, and by Taylor expansion for every $X$ with $\L_X\mu_0=0$. Thus $L$ is natural in $U_x$ (constant coefficients were shown above).

\smallskip\noindent\emph{Step 2 (symbol).} By Proposition \ref{old8prop}(a) the symbol $c(\xi)=\sum_\al C_\al\xi^\al$ of $L$ satisfies \eqref{eq:symbolcond} for all $d\ge1$. For $d=1$ this is the $\on{SL}(m)$-invariance that was previously obtained through the commutator identity; by Proposition \ref{old8prop}(b) it implies that $c$ is a combination of products of brackets, possibly with repeated slots. This is all that $\on{SL}(m)$ gives: a slot may have order $\ge2$ , and the conditions $d\ge2$ in \eqref{eq:symbolcond} are the genuinely nonlinear ones. They are satisfied by all operators in the semi-ring generated by the $J_S$ (Proposition \ref{old8prop}(d)).

\smallskip\noindent\emph{Step 3 (globalization).} For overlapping good sets $U_x,U_{x'}$ in $M\setminus\p M$ the operators $L$ determined by $G$ coincide on functions supported in the overlap, by the uniqueness of the representation in \cite[Thm.~2.3.5]{Hor1983}. Since $M\setminus\p M$ is connected, they define one differential operator $L$ on $M\setminus\p M$ which is natural (Step 1). In a boundary chart $\mu_0=du$ and the coefficients of $L$ are the constants $C_\al$ found on the dense open set $V$ below, so $L\in\mc N_{n-1}(M,\mu_0)$. The term $C\int f_1\cdots f_n\mu_0$ is the case $L=Cf_1\cdots f_{n-1}$.

This proves the interior part of the Lemma for $x\in M\setminus\p M$.

\smallskip\noindent\textbf{The boundary part.}

\smallskip\noindent\emph{Step B0 (subtracting the interior).} Let $L$ be the natural operator found above and $G_{\mathrm{int}}(f)=\int_ML(f_1,\dots,f_{n-1})f_n\mu_0$. If $X$ is divergence free and tangent to all strata then, by naturality of $L$ and Stokes' theorem on manifolds with corners, $\sum_jG_{\mathrm{int}}(\dots,\L_Xf_j,\dots)=\int_M X\big(L(f)f_n\big)\mu_0=\int_{\p M}\iota_X(L(f)f_n\mu_0)=0$, since $\iota_X\mu_0$ vanishes on vectors tangent to $\p M$. So $G_\p:=G-G_{\mathrm{int}}$ is invariant. In a chart $(U,u)$ its kernel coefficients are $A_\al-C_\al\mu_0$, where $C_\al$ are the constant coefficients of $L$. On $V=U\cap\{u^1>0,\dots,u^p>0\}$ the interior argument gives $A_\al|_V=C_\al\mu_0$ (the same constants, since $V$ is connected and open in $M\setminus\p M$). So $G_\p$ is supported over $U\cap\p M$.

\smallskip\noindent\emph{Step B1 (one stratum).}
Let $S$ be a connected component of the open codimension-$p$ stratum and choose an adapted
volume chart.  By the structure theorem for distributions supported on coordinate
subspaces, and by translation invariance in the tangential variables, the quotient of
$G_\p$ by terms supported on deeper strata is a constant-coefficient local differential
functional on $S$.  If $p<m$, Proposition \ref{boundaryfactorization} therefore writes it as
a finite sum of the expressions \eqref{eq:boundaryfactorization}.  The coefficients glue on
overlaps because the scalar factors are restrictions of ambient natural operators and the
exterior factor is an honest top-degree form on $S$.  If $p=m$, the stratum is a vertex and
the same structure theorem gives directly a vertex-natural scalar operator
\eqref{eq:vertexgenerator}; no horizontal integration by parts is available or needed.

\smallskip\noindent\emph{Step B2 (stratified subtraction).}
Order the strata by codimension.  Having treated all strata of codimension $<p$, subtract
the corresponding globally defined invariant integrals.  The residual distribution is
supported in codimension at least $p$.  Apply Step B1 on each connected component of the
open codimension-$p$ stratum and subtract the resulting integrals.  The new residual
vanishes there and hence is supported in codimension at least $p+1$.  Iterating to $p=m$
proves the boundary statement.  No choice of a boundary measure occurs at any stage: the
integrands being subtracted are top-degree forms on the relevant strata.
\end{proof}

\subsection{Completion of the generator argument}\label{old9}
We record explicitly the reduction from partial diagonals to the full diagonal.  For a
partition $\pi=\{B_1,\ldots,B_r\}$ of $\{1,\ldots,n\}$ let $\Delta_\pi\subset M^n$ be the
corresponding partial diagonal, and let $\Delta_\pi^\circ$ be the subset on which the $r$
block-points are pairwise distinct.  A partition $\rh$ is finer than $\pi$ when
$\Delta_\rh\subsetneq\Delta_\pi$.

\begin{lemma}[Factorization on an open partial diagonal]\label{partialfactorization}
Let $K\in\mathcal D'(M^n)$ be $\Diff_0(M,\mu_0)$-invariant and supported on
$\Delta_\pi$.  Modulo invariant distributions supported on the union of the proper finer
diagonals $\Delta_\rh\subsetneq\Delta_\pi$, the restriction of $K$ to
$\Delta_\pi^\circ$ is a finite sum of external products of invariant full-diagonal
functionals, one for each block $B_a$.  Equivalently, its class in the quotient by finer
diagonal support belongs to the algebra generated by the full-diagonal invariant tensor
fields of arities $|B_1|,\ldots,|B_r|$.
\end{lemma}

\begin{proof}
Fix a point of $\Delta_\pi^\circ$, represented by pairwise distinct block-points
$x_1,\ldots,x_r$.  Choose pairwise disjoint coordinate neighborhoods $U_a$ of $x_a$.
On $U_1^{B_1}\times\cdots\times U_r^{B_r}$ the normal bundle of $\Delta_\pi^\circ$
splits as the direct sum of the normal bundles of the full diagonals in the factors
$U_a^{B_a}$.  The structure theorem for distributions supported on a submanifold therefore
writes $K$ locally as a finite sum of products of finite normal derivatives in the
individual blocks, with coefficient distributions on $U_1\times\cdots\times U_r$.

The subgroups of exact volume-preserving diffeomorphisms supported in the mutually
disjoint $U_a$ commute, 
\[
G_1\times\cdots\times G_r \subset \Diff_0(M,\mu_0),\qquad
G_a=\Diff_{\mathrm{exact},c}(U_a,\mu_0),
\]
and act independently on the finite-dimensional normal-jet factors. 
Therefore
\[
(V_1\otimes\cdots\otimes V_r)^{G_1\times\cdots\times G_r}
=
V_1^{G_1}\otimes\cdots\otimes V_r^{G_r}.
\]
After these invariant normal-jet factors have been extracted, the remaining scalar coefficient on the locus of distinct interior block-points is constant by finite \(r\)-transitivity of the compactly supported volume-preserving group.
Applying this successively to the finite normal-jet expansion shows that its class is a finite sum of external products of invariant full-diagonal block functionals.
 On the locus of distinct
interior block-points, finite $r$-transitivity of the compactly supported
volume-preserving group makes this coefficient constant.  The same conclusion on a fixed
boundary-stratum type follows locally from the corresponding corner pseudogroup; changes
of stratum type lie in deeper boundary support and are treated by the stratified induction
of Lemma \ref{old8}.

The construction is independent of the chosen disjoint neighborhoods: two choices agree
on their common domain by locality.  The only locus where such a disjoint-neighborhood
separation ceases to exist is where two block-points collide.  That locus is exactly the
union of the proper finer diagonals.  Hence the local external-product description glues
in the quotient by distributions supported on those finer diagonals, proving the lemma.
\end{proof}

\begin{proof}[Completion of the proof of Main Theorem \ref{maintheorem}]
Lemma \ref{old6} first removes the finite sum $P$ of decomposable products of the invariant
one-tensors $\int_M f\mu_0$ and $f(v)$, $v\in\p^mM$.  The remaining kernel is supported on
the union of the proper partial diagonals.  Choose a maximal partial diagonal in its
support.  Lemma \ref{partialfactorization} expresses its class, modulo finer diagonals, as
a finite sum of products of full-diagonal invariant tensor fields.  Subtract these
products.  The remainder is supported on a union of strictly finer diagonals.  Since the
partition lattice is finite, descending induction terminates at the full diagonal.

Lemma \ref{old8} classifies the full-diagonal contribution.  In the interior it is an
integral of a $\mu$-natural scalar differential operator against $\mu$.  On a positive-dimensional
codimension-$p$ stratum Proposition \ref{boundaryfactorization} identifies it, modulo
deeper strata, with a finite sum of the explicit generators \eqref{eq:boundaryfactorization};
at depth $m$ it gives the vertex generators \eqref{eq:vertexgenerator}.  The stratified
subtraction in Lemma \ref{old8} terminates after depth $m$.  Restoring the
Moser reduction and allowing the coefficients to depend smoothly on the total volume
$\mu(M)$ gives precisely the generators stated in Main Theorem \ref{maintheorem}.
\end{proof}

\section{Adding contravariant tensor fields} \label{sec:contra}

In this section the cotangent spaces are endowed with their natural strong dual topology.  Thus, after fixing
$\mu\in\Dens_+(M)$ and identifying $T_\mu\Dens_+(M)$ with $C^\infty(M)$ by
$f\mapsto f\mu$, we identify
\[
 T_\mu^*\Dens_+(M)\cong C^\infty(M)'.
\]
We continue to assume $\dim M=m\ge2$.  We use the fact that the group of compactly supported
$\mu$-preserving diffeomorphisms of the connected interior $M^\circ$ acts $n$-transitively on
ordered $n$-tuples of pairwise distinct points, for every finite $n$; see
\cite{MichorVizman94}.  Such diffeomorphisms extend by the identity to elements of
$\Diff_0(M,\mu)$.

\begin{proposition}\label{prop:contra}
For a manifold $M$ with corners, any bounded purely contravariant $n$-tensor field which is invariant under the action of $\Diff_0(M)$ has the form
\begin{gather*}
L: T^*\Dens_+(M) \x_{\Dens_+(M)} \dots \x_{\Dens_+(M)} T^*\Dens_+(M) \to \mathbb R,
\\
L_\mu (A_1,\dots,A_n) = F(\mu(M))\, A_1(\mu)\cdots A_n(\mu)
\end{gather*}
for a smooth function $F:\mathbb R_{>0}\to \mathbb R$.
\end{proposition}

\begin{proof}
Fix first a probability density $\mu_0$.  Then  $C^\infty(M)\cong  T_{\mu_0}\Dens_+(M)$ via $f\mapsto f\mu_0$ and hence also
$T^*_{\mu_0}\Dens_+(M)\cong C^\infty(M)'$.
Since $C^{\infty}(M)$ is nuclear Fr\'echet, the kernel theorem identifies bounded $n$-linear forms on
$(E')^n$ with smooth kernels on $M^n$.  Thus there is a unique
$K\in C^\infty(M^n)$ such that
\[
 L_{\mu_0}(A_1,\dots,A_n)
 =\langle A_1\otimes\cdots\otimes A_n,K\rangle,
\]
where in the pairing we regard $A_i$ as a distribution on functions by
$A_i(f):=A_i(f\mu_0)$.

For $\varphi\in\Diff_0(M,\mu_0)$, invariance of $L_{\mu_0}$ is equivalent to
\[
 K(\varphi(x_1),\dots,\varphi(x_n))=K(x_1,\dots,x_n).
\]
By \cite{MichorVizman94}, the subgroup consisting of compactly supported
$\mu_0$-preserving diffeomorphisms of $M^\circ$ is $n$-transitive on the configuration
space
\[
 (M^\circ)^{(n)}=\{(x_1,\dots,x_n):x_i\ne x_j\text{ for }i\ne j\}.
\]
Hence $K$ is constant on $(M^\circ)^{(n)}$.  This set is dense in $M^n$, and $K$ is
smooth, so $K\equiv c$.  Therefore
\[
 L_{\mu_0}(A_1,\dots,A_n)=c\,A_1(\mu_0)\cdots A_n(\mu_0).
\]

The same argument applies to $c\mu_0$, $c>0$.  Moser's theorem
(Theorem~\ref{Moser}) shows that two positive densities with the same total volume lie in
the same $\Diff_0(M)$-orbit for the present purpose, so the coefficient depends only on
$c=\mu(M)$.  Along the smooth curve $c\mapsto c\mu_0$ the tensor field is smooth; evaluating
it on the smooth covectors
\[
 A_c(\alpha):=\frac{1}{c}\int_M\alpha,
 \qquad A_c(c\mu_0)=1,
\]
shows that this coefficient is a smooth function $F(c)$.  Transporting the formula from
$c\mu_0$ to an arbitrary $\mu$ gives the assertion.
\end{proof}

\begin{proposition}\label{prop:co-contra}
Let  $M$ be a manifold with corners. Then any bounded $(1,1)$-tensor field on $\Dens_+(M)$ invariant
under $\Diff_0(M)$ has the form
\begin{align*}
L_\mu(A,\alpha)
={}&F(\mu(M))\,A(\alpha)
 +G(\mu(M))\,A(\mu)\int_M\alpha \\
&\quad +\sum_{v\in \p^m M}H_v(\mu(M))\,A(\mu)
       \left(\frac{\alpha}{\mu}\right)(v),
\end{align*}
where $F,G,H_v\in C^\infty(\mathbb R_{>0},\mathbb R)$.  If $\p^mM=\varnothing$, the last
sum is absent.
\end{proposition}

\begin{proof}
Again fix a probability density $\mu_0$ and write $\alpha=f\mu_0$.  For each
$f\in C^\infty(M)$, the map
\[
 A\longmapsto L_{\mu_0}(A,f\mu_0)
\]
is a continuous linear functional on the strong dual $C^\infty(M)'$.  Since
$C^\infty(M)$ is reflexive, there is a unique smooth function $Kf$ such that
\[
 L_{\mu_0}(A,f\mu_0)=A((Kf)\mu_0).
\]
Boundedness of $L$ makes $K:C^\infty(M)\to C^\infty(M)$ a continuous linear operator.
The $\Diff_0(M,\mu_0)$-invariance of $L$ is exactly the equivariance relation
\[
 K(f\circ\varphi)=(Kf)\circ\varphi,
 \qquad \varphi\in\Diff_0(M,\mu_0).
\]

We classify such operators.  For $x\in M^\circ$ let
\[
 \ell_x(f):=(Kf)(x).
\]
Then $\ell_x$ is a distribution invariant under the stabilizer
$G_x=\{\varphi\in\Diff_0(M,\mu_0):\varphi(x)=x\}$.  The same local argument as in
Lemma~\ref{old4}, now applied away from $x$, shows that on
$M\setminus(\{x\}\cup  \p^m M)$ the distribution $\ell_x$ is a constant multiple of $\mu_0$.
Indeed the exact divergence-free vector fields used there may be chosen with support away
from $x$.  At a vertex, the vertex part of Lemma~\ref{old4} shows that only a multiple of
$\delta_v$ can occur.  Finally, at $x$ the linear isotropy contains $\on{SL}(T_xM)$; invariance
under it eliminates all derivatives of $\delta_x$, leaving only a multiple of
$\delta_x$.  Consequently
\[
 \ell_x=c_x\delta_x+a_x\mu_0+\sum_{v\in  \p^m M}b_{v,x}\delta_v.
\]

The $2$-transitivity of the compactly supported volume-preserving diffeomorphism group on
$M^\circ$ implies that the stabilizer of one interior point is transitive on the remaining
interior points.  Equivariance of the family $x\mapsto\ell_x$ therefore makes the
coefficients $c_x,a_x,b_{v,x}$ independent of $x\in M^\circ$; write them as
$c,a,b_v$.  Hence
\[
 (Kf)(x)=c f(x)+a\int_Mf\,\mu_0+\sum_{v\in  \p^m M}b_v f(v),
 \qquad x\in M^\circ.
\]
Both sides are smooth in $x$, and $M^\circ$ is dense in $M$, so this identity holds on all
of $M$.  Substitution into the defining pairing gives
\[
 L_{\mu_0}(A,f\mu_0)
 =cA(f\mu_0)
 +aA(\mu_0)\int_Mf\,\mu_0
 +\sum_{v\in  \p^m M}b_vA(\mu_0)f(v).
\]
Equivalently,
\[
 L_{\mu_0}(A,\alpha)
 =cA(\alpha)+aA(\mu_0)\int_M\alpha
 +\sum_{v\in  \p^m M}b_vA(\mu_0)\left(\frac{\alpha}{\mu_0}\right)(v).
\]

Apply the same argument at $c\mu_0$.  By Moser's theorem the coefficients depend only on
$c=\mu(M)$, and smoothness of the tensor field along $c\mapsto c\mu_0$ makes them smooth
functions of $c$.  Since every $\varphi\in\Diff_0(M)$ fixes each zero-dimensional stratum
component $v\in  \p^m M$, all displayed terms are invariant.  Transport by Moser's
diffeomorphism gives the stated formula for arbitrary $\mu$.
\end{proof}

\subsection{Remark}\label{all-tensors} The Main Theorem \ref{maintheorem}, Propositions \ref{prop:contra} and the generalized contraction in Proposition \ref{prop:co-contra} suggest the general form of a tensor field on $\Dens_+(M)$.


\end{document}